\documentclass[11pt]{amsart}

\usepackage{enumitem}
\usepackage{verbatim}
\usepackage{hyperref}
\usepackage{graphicx}
\usepackage{xcolor}
\usepackage{shellesc}
\usepackage{amssymb}
\usepackage{amsmath}
\usepackage{amsthm}

\newcommand{\N}{\mathbb{N}}
\newcommand{\R}{\mathbb{R}}
\newcommand{\Z}{\mathbb{Z}}

\newcommand{\AAA}{\mathcal{A}}
\newcommand{\BBB}{\mathcal{B}}
\newcommand{\CC}{\mathcal{C}}
\newcommand{\EE}{\mathcal{E}}
\newcommand{\FFF}{\mathcal{F}}

\newcommand{\MMM}{\mathcal{M}}
\newcommand{\PPP}{\mathcal{P}}

\newcommand{\di}{\partial}
\newcommand{\wM}{\widetilde{M}}
\newcommand{\ideal}{\di\widetilde M}

\newcommand{\htop}{h_{\mathrm{top}}}
\newcommand{\hvol}{h_{\mathrm{vol}}}
\newcommand{\dididi}{ (\partial \wM \times \partial \wM) \setminus \diag}
\newcommand{\sqbd}{\di^2\wM}
\newcommand{\tmu}{\widetilde{\mu}}
  \newcommand{\wE}{\mathcal{\widetilde E} }
\newcommand{\wti}{\widetilde }

\newcommand{\inj}{\mathrm{inj}}

\newcommand{\grad}{\mathrm{grad}}

\newcommand{\card}{\mathrm{card}}

  \newcommand{\diam}{\mathrm{diam}}
  \newcommand{\cl}{\mathrm{cl}}
  \newcommand{\supp}{\mathrm{supp}}
  \newcommand{\diag}{\mathrm{diag}}

 \newcommand{\Id}{\mathrm{Id}}
  \newcommand{\End}{\mathrm{End}}

\newtheorem{MainThm}{Theorem}
\newtheorem{theorem}{Theorem}[section]
\newtheorem{lemma}[theorem]{Lemma}
\newtheorem{proposition}[theorem]{Proposition}

\newtheorem{definition}[theorem]{Definition}

\newtheorem*{thma*}{Theorem}

\theoremstyle{remark}
\newtheorem{remark}[theorem]{Remark}

\DeclareMathOperator{\Leb}{Leb}

\DeclareMathOperator{\id}{Id}

\DeclareMathOperator{\tr}{tr}

\numberwithin{equation}{section}
\numberwithin{figure}{section}

\begin{document}
	
\UseRawInputEncoding

\title[]{Uniqueness of the measure of maximal entropy for geodesic flows on coarse hyperbolic manifolds without conjugate points}
\author{Gerhard Knieper}
\address{Faculty of Mathematics, Ruhr University Bochum, 44780 Bochum, Germany}
\email{gerhard.knieper@rub.de}
\date{\today}
%\dedicatory{Dedicated to Manfred Denker on his 80th birthday}
\begin{abstract}

In this article we study geodesic flows on closed Riemannian manifolds without conjugate points 
and divergence property of geodesic rays.
If the fundamental group is Gromov hyperbolic and residually finite 
we prove,  under appropriate assumptions on the expansive set, that the geodesic flow has a unique measure of maximal entropy. This generalizes  corresponding results of Climenhaga, Knieper and War \cite{CKW21} proved under the stronger assumption
of the existence of a background metric of negative sectional curvature. Using results from \cite{CKW22}
and \cite{CKW21} we obtain that the measure of maximal entropy is given by the limiting distribution of closed orbits.
Furthermore, a Margulis type estimate on the number of free homotopy classes containing a closed geodesic of period smaller than $t >0$ follows.
\end{abstract}

\thanks{The author was partially supported by the German Research Foundation (DFG),
CRC TRR 191, \textit{Symplectic structures in geometry, algebra and dynamics.}
}
\maketitle

\section{Introduction}\label{sec:intro}
If $f: X \to X$ is a homeomorphism of a compact metric space $X$,
the topological entropy $h_{\mathrm{top}}(f)$ of $f$ is an invariant in topological dynamics measuring the orbit complexity of the dynamical system on an
exponential scale. On the other hand, measure-theoretic entropy $h_{\mu}(f)$ is an invariant of 
measurable dynamics reflecting the average complexity of the system relative
to  a measure $\mu$ contained in the space of $f$-invariant Borel probability measure $\MMM_f(X)$ on $X$. 
The relation between these notions is provided by the variational principle which implies (see e.g. \cite{pW82} for details)
$$
h_{\mathrm{top}} (f) = \sup \{h_\mu (f) \;|\; \mu \in \MMM_f(X)\}.
$$
A measure $\mu \in \MMM_f(X)$ such that $h_{\mu}(f) =h_{\mathrm{top}}(f)$ is called a measure of maximal entropy (MME).
For continuous flows $\phi^t: X \to X$ topological - and measure theoretic entropy is defined as the entropy of the time 1 map $\phi := \phi^1$.

In the Riemannian setting topological entropy of the geodesic flow is related to large scale geometry.
If $(M,g)$ is a closed Riemannian manifold without conjugate points 
(all geodesics are globally  minimizing if lifted to the universal cover $\wM$) then by a theorem of Freir\'{e}  and  Ma\~n\'e  \cite{FM82} the topological entropy  $\htop (\phi)$ of the geodesic
flow $\phi^t: SM \to SM$ agrees with the volume entropy $\hvol(g)$.  Volume entropy measures the asymptotic volume growth of geodesics balls $B(p,r) \subset \wM$ in the universal covering $\wM$ of $M$  and is given by
$$
\hvol(g) = \lim\limits_{r \to \infty} \frac{\log \mathrm{vol} B(p,r)}{r} 
$$
For manifolds $(M,g)$ of non-positive sectional curvature, which are special manifolds of no conjugate points, 
an easier proof of  the result of Freir\'{e}  and  Ma\~n\'e  were previously given by Manning \cite{aM79}. However, in contrast to non-positive sectional curvature,  no conjugate points 
is  not a local condition. Properties like the convexity of distance functions which
 hold for  manifolds of non-positive sectional curvature and even no focal points
are not anymore true for manifolds without conjugate points.  Furthermore, the flat strip theorem, important in the study of manifolds 
of  non-positive sectional curvature, fail under the assumption of no conjugate points \cite{kB92}.

In this paper, we like to study the entropy of geodesic flows on closed Riemannian manifolds $(M, g)$ without conjugate points under some coarse hyperbolicity of the geodesic flow.
Hyperbolicity implies the existence of a non-trivial expansive set $\EE \subset SM$ consisting of flow lines
which diverge from each other over time. More formally, if $d$ is the distance function induced by the Riemannian metric on $M$, a vector  $v \in SM$ is contained in $\EE$ if there is some
 $\epsilon >0$ such that for all $w \in SM $ with $d(c_v(t), c_w(\R)) < \epsilon$ for all $t \in \R$,
the geodesics $c_v$, $c_w$ with initial conditions $v$ and $w$ agree up to a time shift. Furthermore, the expansive set
is invariant under the geodesic flow.

Geodesic flows are called expansive if $SM =\EE$.
In case that $(M,g)$ has no conjugate points it was proved by Ruggiero  \cite{rR94} that the expansiveness of the geodesic
flow 
implies that the fundamental group $\pi_1(M)$ is Gromov hyperbolic. Furthermore the divergence property hold, i.e. for each pair of geodesic rays $c_1, c_2: [0, \infty) \to \wM $ on the universal cover we have 
$\sup_{t \ge 0} d(c_1(t),c_2(t)) = \infty$. 
While negatively curved manifolds have expansive geodesic flows, this is not longer true for a typical manifold without conjugate points or even non-positive curvature. Nevertheless, in many cases the expansive set is not empty.

In the following we call an invariant Borel probability measure $\nu \in \MMM_\phi(SM)$ of the geodesic flow $\phi^t: SM \to SM$ non-expansive if its weight of the expansive set is vanishing, i.e. $\nu(\EE) =0$. It is natural to assume that such measures have less complexity and are not of maximal entropy. Using this assumption we can prove the following.

\begin{MainThm}\label{thm:main}
 Let $(M,g)$ be a closed smooth $(C^\infty)$ Riemannian manifold without conjugate points with divergence property of geodesic rays and Gromov hyperbolic and residually finite fundamental group.
  Assume that 
  $$
 h_\nu(\phi)<\htop(\phi)
 $$
for all non-expansive measures
$\nu \in \MMM_\phi(SM)$.   Then the geodesic flow has a unique measure $\mu$ of maximal entropy.  
 
 Furthermore, $\mu(\EE) =1$, $\mu$  is mixing and fully supported on the unit tangent bundle.

\end{MainThm}
\begin{remark}
\begin{itemize}
\item Closed Riemannian manifolds without conjugate points and Gromov hyperbolic fundamental group have positive topological
entropy (see \cite{CK02}  and Remark  \ref{rem:conj-Grhyp}). 
\item It follows from the work of Newhouse \cite{sN89} that for smooth geodesic flows $\phi^t$
there exists a measure $\nu \in \MMM_\phi(SM)$ of maximal entropy, i.e. $h_\nu(\phi)= \htop(\phi)$.
Hence, the assumption of our theorem forces  $\nu(\EE)$ to be positive. In particular, the expansive set is not empty.
\item If $(M,g)$ is a closed and non-flat  surface without conjugate points, then all assumptions of Theorem \ref{thm:main}
hold (see \cite{CKW21}). Therefore such surfaces have a unique MME.
\item If $(M,g)$ is a closed smooth Riemannian manifold without conjugate points and expansive geodesic flow (i.e $\EE =SM$)
 the uniqueness of the MME has been obtained by Bosch\'e in \cite{aB18}.
As we mentioned above a closed manifold without conjugate points and expansive geodesic flow has the divergence
property and Gromov hyperbolic fundamental group. Since for expansive geodesic flows non-expansive measures obviously do not exist,   the result of Bosch\'e follows from our theorem as a special case.
\item If $(M,g)$ is a closed rank 1 manifold of non-positive curvature the rank 1 (regular)  set consists of orbits of the geodesic flow which do not have non-trivial parallel Jacobi-fields orthogonal to the geodesic (see e.g. \cite{gK97} or \cite{gK98}).
It is a consequence of the flat strip theorem that the regular set is contained in the expansive set.
In  \cite{gK98} we showed that for closed rank 1 manifolds the geodesic flow has a unique MME. Furthermore, the measure has full weight
on the regular, and hence, on the expansive set. A different proof of the uniqueness of the MME was later given in  \cite{BCFT}.
Since there are examples of closed rank 1 manifolds whose fundamental group has $(\Z^2,+)$ as a subgroup, the 
fundamental group is generally not Gromov hyperbolic (see \cite{gK98}).
\end{itemize}
\end{remark}
In \cite{CKW21}  Climenhaga, War and the author proved Theorem \ref{thm:main} under the more special condition of a background
metric of negative curvature and 
the slightly stronger entropy gap assumption 
$$
 \sup\{h_\nu(\phi) : \nu\in \MMM_\phi(SM), \nu(\EE)=0\} < \htop(\phi).
 $$
 for non-expansive measures.
 The entropy gap would follow from our condition that non-expansive measure do not have maximal
 entropy provided the expansive set is open.  Namely, if $\EE$ is open and the entropy gap would not hold there would exist a sequence of measures 
 $\nu_n \in \MMM_\phi(SM)$   with $\nu_n(\EE)=0 $ converging weakly to $\nu$ and $\lim_{n \to \infty}h_{\nu_n}(\phi) = \htop(\phi)$. Since  $\EE$ is open, $\nu(\EE) \le \liminf_{n \to \infty}\nu_n(\EE)= 0$ and by upper semi-continuity of entropy this would yield $h_{\nu}(\phi) = \htop(\phi)$. But this contradicts our assumption that  $h_{\nu}(\phi) <\htop(\phi)$ for non-expansive measures.
 However, to our knowledge openness of the expansive set is not known in our setting.

The assumption of a background
metric of negative curvature implies that the fundamental group is Gromov hyperbolic (see subsection \ref{subsec:hypgroups}). 
Due to the uniformization theorem for surfaces and the proof of the geometrization conjecture in dimension three
Gromov hyperbolicity of the fundamental group implies the existence of a metric with even constant negative curvature.
However,  for closed manifolds of dimension bigger than three,  Gromov hyperbolicity of the fundamental group and existence of a metric without conjugate points might not be enough to yield a metric of negative curvature. In any case, to provide a solution to this question is a very difficult problem.

In \cite{CKW21}  the proof of the uniqueness of the measure of maximal entropy used the background metric  of negative curvature
to establish with the help of the Morse Lemma a coarse specification property  
for the geodesic flow.  Applying the work of Climenhaga and Thompson \cite{CT16}, the specification property was used
to prove the uniqueness of the MME.
 
 However, the proof of the above theorem does not require specification but relied on methods derived in a paper
 of the author  \cite{gK98} on the uniqueness of the measure of maximal entropy for geodesic flows on non-positively curved rank 1 manifolds.\\
 
 There is a interesting and quite flexible notion due to Bowen \cite{rBo72}, called entropy expansiveness  which hold for many dynamical systems for which expansiveness fails (see section \ref{sec:e-expansiveness}).
 
 Given a closed Riemannian manifold $(M,g)$ without conjugate points and   $(\wM,g)$ 
 be the universal cover with the lifted Riemannian metric denoted again by $g$. For $v \in S\wM$ and $\rho >0$ we define 
 the set
 $$
 Z_{\rho} (v) = \{w \in S\wM \mid d(c_v(t), c_w(t)) \le \rho, t \in \R \}.
$$
 The geodesic flow $\phi^t: SM \to SM$  is called entropy expansive at scale $\rho >0$
if $\htop(\tilde \phi, Z_\rho(v))  = 0$, where $\tilde \phi^t$ is the geodesic flow lifted to $S \wM$.
Due to the flat strip theorem geodesic flows on manifolds of non-positive curvature or more generally no focal points are entropy expansive at any scale (see \cite[proposition 3.3]{gK98}).
This also holds for non-flat surfaces without conjugate points \cite[Lemma 4.5]{GKOS14} even so the flat strip theorem fails in this case \cite{kB92}. As far as we know, there is no example of a closed Riemannian manifold with metric without conjugate known for which the
geodesic flow is not entropy expansive for all or even some $\rho >0$. Alternatively to Theorem \ref{thm:main}, we obtain
the uniqueness of the MME under the following conditions.
\begin{MainThm}\label{thm:expansiveMME}
 Let $(M,g)$ be a closed Riemannian manifold without conjugate points with divergence property and Gromov
 hyperbolic and residually finite fundamental group. Assume that the geodesic flow is entropy expansive at a scale larger than $8 \delta$, where $\delta$ is the Gromov hyperbolicity constant of the universal cover. If the expansive set has non-trivial interior then the geodesic flow
 has a unique measure of maximal entropy $\mu$.
 
Furthermore, $\mu(\EE) =1$, $\mu$  is mixing and fully supported on the unit tangent bundle.
\end{MainThm}
%\begin{remark}
%In particular, in comparison to Theorem \ref{thm:main}, the assumption of the theorem above implies that the entropy of all non-expansive measure are strictly smaller than the
%topological entropy. 
%\end{remark}

 Using the results of Climenhaga, War and the author proved in \cite[Theorem 1.2]{CKW21} 
 together with Theorem \ref{thm:conj-Grhyp} and Remark \ref{rem:conj-Grhyp},
 we can conclude that the measure of maximal entropy is given by the limiting distribution of closed orbits.
 
Furthermore, as in \cite[Theorem 1.2]{CKW22}, an estimate on the growth of pairwise non-free-homotopic closed geodesics,
obtained by Margulis \cite{gM69}  in the case of negative curvature, follows. More precisely:

 \begin{MainThm}\label{thm:closed geodesics}
	Let $(M,g)$ be a closed Riemannian manifold such that the assumption in Theorem \ref{thm:main} 
	or Theorem \ref{thm:expansiveMME} hold. Denote for $T>0$  by $\PPP(T)$ be any maximal set of pairwise non-free-homotopic closed geodesics of minimal length in the free homotopy classes and $P(T) =\card \PPP(T) $ its cardinality. Consider the measures
\begin{equation*}\label{eqn:nut}
\mu_T= \frac {1}{P(T)} \sum_{c\in \PPP(T)} \frac{\Leb_c}{T},
\end{equation*}
where $\Leb_c$ is Lebesgue measure (length) along the curve $\dot{c}$ in the unit tangent bundle $SM$. Then
\begin{enumerate}
\item \label{eqn:distr}

The measures $\mu_T$ converge  in the weak* topology  as $T\to \infty$ to the measure of maximal entropy.

\item
Furthermore, 
\begin{equation*}\label{eqn:margulis}
P(T) \sim \frac{e^{hT}}{hT},
\end{equation*}
which means that the ratio of $P(T)$ and  $\frac{e^{hT} }{hT} $ converges $1$ as $T \to \infty$.
\end{enumerate}
		
\end{MainThm}

%The paper is organized as follows: 
%{\bf Acknowledgement.} The authors would like to thank the anonymous referees for a careful reading of this article.

\section{Gromov hyperbolic manifolds without conjugate points }
In the next subsection we recall basic facts for Gromov hyperbolic geodesic metric spaces
introduced by Gromov in his seminal work \cite{mG87}. For a introduction see also \cite{BH99}. 

\subsection{Basic facts on Gromov hyperbolic metric spaces}
\begin{definition} \label{def:geod.metric}
Let $(X, d)$ be a metric space. A curve $c: I \to X$ defined on an interval $I \subset \mathbb{R}$
is called a geodesic, if $c$ is an isometry, i.e., $d(c(t), c(s))= |t-s| $ for $t, s \in I$.
A geodesic metric space $(X, d)$ is a  metric space, where each pair of points can be joint by a geodesic.
\end{definition}
\begin{remark}
Note, that in Riemannian geometry geodesics are local isometries. Geodesics in the sense of metric spaces
correspond to minimal geodesics in the Riemannian setting.
\end{remark}
There are several equivalent definitions of Gromov hyperbolicity. The most common definition is the following.
\begin{definition} \label{def:gomovhyp}
Let $\delta$ be a non-negative number. A proper geodesic metric space  $(X, d)$ is called $\delta$-hyperbolic if all geodesic triangles are $\delta$-thin, i.e., each
side of a geodesic triangle is contained in the $\delta$-neighborhood of the two other sides.
A geodesic  metric space is called Gromov hyperbolic if it is $\delta$-hyperbolic for some $\delta \ge 0$.
\end{definition}
A fundamental feature is that Gromov hyperbolicity  is invariant under quasi isometries.
 
\begin{definition} \label{def:quasi-isometry}
A pair of metric spaces $(X_1, d_1)$  and $(X_2, d_2)$ are called quasi-isometric if there exist
constants $\lambda \ge 1$, $\alpha \ge 0$ and $ C \ge 0$ and a map 
$
f: X_1 \to X_2
$
such that
$$
\frac{1}{\lambda} d_1(x,y) -\alpha \le d_2(f(x), f(y)) \le \lambda d_1(x,y) +\alpha
$$
and $d_2(y, f(X_1)) \le C$ for all $y \in X_2$. Furthermore, $f$ is called a quasi-isometry.
\end{definition}

\begin{theorem} \label{thm:quasi-isometry}
Let $(X_1, d_1)$  and $(X_2, d_2)$ be quasi-isometric geodesic metric spaces.
Then $X_1$ is Gromov hyperbolic if and only if $X_2$ is Gromov hyperbolic.
\end{theorem}
  \begin{proof} For a proof see e.g. Theorem 1.9 on page 402 of  \cite{BH99}.
  \end{proof}

The following simple Lemma will be frequently used and can be viewed as the coarse version of expansivity.
\begin{lemma}\label{lem:endpts-suffice}
Let $(X, d)$ be $\delta$-hyperbolic geodesic metric space and $\rho >0$. Consider two geodesics
 $c_1,c_2\colon [0,T]\to X$ with 
 $d(c_1(0),c_2(0)) \leq \rho $ and $d(c_1(T),c_2(T)) \leq \rho$. Then 
  $d(c_1(t),c_2(t)) \leq 4 \delta + 3 \rho $ for all $ t\in [0,T]$.
  \end{lemma}
  \begin{proof}
  Let
 $c_1,c_2\colon [0,T]\to X$ be two geodesics with 
 $d(c_1(0),c_2(0)) \leq \rho $ and $d(c_1(T),c_2(T)) \leq \rho$.
  Consider geodesics $c \colon [0,T^\prime] \to X$, $\alpha \colon [0,a] \to X$ and  $\beta: [0,b] \to X$  with $c(0) = c_1(0)$ and $c(T^\prime] ) = c_2(T)$, $\alpha(0) = c_1(0)$,  $\alpha(a) = c_2(0)$ and $\beta(0) = c_1(T)$,  $\beta(b) = c_2(T)$.
 
 Fix $t \in [0,T]$ and assume first that $ T > 2(\delta + \rho)$ and $t \in ( \delta + \rho, T- \delta - \rho)$ holds.
  Since $(X, d)$ is a $\delta$-hyperbolic geodesic metric space we have that
  $$
  d(c_1(t), c[0, T^\prime] \cup  \beta[0, b]) \le \delta 
  $$
  holds.
  By the choice of $t$ we obtain  $ d(c_1(t), \beta[0, b]) \ge T-t -\rho > \delta$. Consequently there
  exists some $ s \in [0,T^\prime]$ with $d(c_1(t), c(s)) ) \le \delta $. Using the triangle inequality
  this yields $| t-s| \le \delta$ and therefore
  $$
  d(c_1(t), c( t) ) \le  d(c_1(t), c( s)) +d(c(s), c(t)) \le 2\delta 
  $$
  Using $\delta$-hyperbolicity again we obtain
   $$
  d(c(t), c_2[0, T] \cup  \alpha[0, a]) \le \delta 
  $$
  Since $d(c(t),\alpha[0, a]) \ge t - \rho $
  and $t - \rho> \delta$ there exists $ s \in [0, T]$ with $d(c(t) , c_2(s)) \le \delta$. Furthermore,
  the triangle inequality yields
  $$
  |t-s| \le \delta + \rho
  $$
  and therefore
  \begin{align*}
  d(c_1(t) , c_2(t)\le &d(c_1(t) , c(t)) + d(c(t) , c_2(s)) + d(c_2(s) , c_2(t))\\
 \le & 2\delta + \delta + \delta + \rho = 4 \delta + \rho
   \end{align*}
In case that  $T \le  2(\delta + \rho)$, $0 \le t \le \delta + \rho $ or $  T- \delta - \rho \le t $ holds
in easy application of the triangle inequality yields $ d(c_1(t) , c_2(t) )\le 2 \delta + 3 \rho$ which finishes the proof of the Lemma.
  \end{proof}
  Each $\delta$-hyperbolic metric space has a natural boundary consisting of an equivalence class of geodesic rays
  \begin{definition}\label{dfn:GB}
  Let $(X, d)$ be $\delta$-hyperbolic  space. Then two geodesic rays 
 $c_1,c_2\colon [0,\infty) \to X$ are called equivalent if $d(c_1(t), c_2(t)$ is bounded. The equivalence class
 of a ray $c: [0,\infty) \to X$  is denoted by $c(\infty)$ and the set of all equivalence classes by $\partial X$.
  \end{definition}
  In the sequel we will need also need the following Lemma.
   \begin{lemma}\label{lem:GB}
    Let $(X, d)$ be $\delta$-hyperbolic  space. Then the following holds:
    \begin{enumerate}
    \item For distinct $\xi, \eta \in \partial X$ there exists a geodesic
   $c: \R \to X$  such that $c(-\infty) = \xi$ and  $c(\infty) = \eta$.
     \item For each $p \in X$  and $\xi \in \partial X$ there exists a ray  $c: [0,\infty) \to X$ with $c(0) =p$ and  $c(\infty) =\xi$.
   For any other geodesic ray  $c^\prime: [0,\infty) \to X$ with $c^\prime(0) =p$ and  $c^\prime(\infty) =\xi$ we have
   $d(c(t), c^\prime(t)) \le 2 \delta$.
 \item If $c_1: [0,\infty) \to X$, $c_2: [0,\infty) \to X$ is a pair of rays such that  $c_1(\infty) = c_2(\infty)$. 
 \begin{enumerate}
 \item Then there exists
 $T >0$ such that $d(c_1(t), c_2(t)) \le 5 \delta$ for all $t \ge T$.
  \item $d(c_1(t), c_2(t)) \le 19\delta + 3d(c_1(0), c_2(0) )$ for all $t \ge 0$.
 \end{enumerate}
\end{enumerate}
 \end{lemma}
   \begin{proof}
   For a proof of (1), (2), (3a) see \cite[Lemma III.H 3.2, Lemma III.H 3.3]{BH99}. The assertion (3b) is a consequence
   of (3a) and Lemma \ref{lem:endpts-suffice}.
\end{proof}
Gromov introduced  in \cite{mG87} a natural topology on $\partial X$ and  on $ \bar X :=X \cup \partial X$  such that  $ \bar X $ is compact (see  also \cite[Def. III.H 3.5]{BH99}). 
We will give a definition in the geometric context in subsection \ref{subsec:Noconjugatepts-Ghyp}.

\subsection{Gromov hyperbolic groups}\label{subsec:hypgroups}
\begin{definition}\label{def:Hypgroups}
Let $G$ be a finitely generated group and denote by $\CC(G, S)$ the Cayley graph of $G$ with respect to some finite generating set $S \subset G$. Then  $G$ is called Gromov hyperbolic if $\CC(G, S)$ equipped with the word metric
is a $\delta$-hyperbolic metric space for some $\delta\ge 0$
\end{definition}
\begin{remark}
\begin{itemize}
\item
If $S_1, S_2$ are finite generating sets of $G$ then the identity $\id: \CC(G, S_1) \to  \CC(G, S_2)$ is a quasi isometry.
In particular, the definition of Gromov hyperbolicity of a finitely generated group does not depend on the choice of the finite generating set.
\item Let $\Gamma$ be a group acting properly, cocompactly and isometrically on some geodesic metric space $X$.
Then $\Gamma$ is finitely generated and for any reference point $x_0 \in X$ and finite set of generators the orbit map $\Gamma \to X$ with
$\gamma \mapsto \gamma x_0$ extends to a quasi isometry between $\CC(G, S)$ and $X$ (see e.g. \cite[Prop. I.8.19]{BH99}).
In particular,  $\Gamma$ is Gromov hyperbolic if and only if $X$ is Gromov hyperbolic.
\end{itemize}
\end{remark}

Assume that $\Gamma$ is a Gromov hyperbolic group acting properly and cocompactly
by isometries on  proper metric space $X$. Consider for  $x\in X$ and $R>0$ the subset of  $\Gamma$ given by
$$
\Gamma_R(x) =  \{\gamma \in \Gamma \mid\; d(x,\gamma x) \leq R\}.
$$
 Then by a result of Coornaert \cite{mCo93} there are  constants $0 <C_1 \le C_2$ and $h>0$ such that
\begin{equation}\label{eqn:co93}
C_1 e^{hR} \le \card \ \Gamma_R(x) \le C_2 e^{hR}
\end{equation}
If $(X,g)$ is a Riemannian manifold then $h$ is the volume entropy $\hvol(g)$, i.e.
$$
h =  \lim\limits_{r \to \infty} \frac{\log \mathrm{vol} B(p,r)}{r} 
$$
Define the translation length of $\gamma \in \Gamma$ by
 \begin{equation} \label{e: length}
\ell(\gamma) = \inf_{x \in X} d(x,\gamma x).
\end{equation}
Note that the infimum in
\eqref{e: length} is attained
for each $\gamma \in \Gamma$ (see
for example \cite[Prop. II.6.10]{BH99}). Furthermore, $\ell(\gamma^{-1}) = \ell(\gamma)$ and
$\ell(\alpha\gamma\alpha^{-1}) = \ell(\gamma)$ for
every isometry $\alpha$ of $X$. Therefore the length of a conjugacy class $[\gamma]$  of $\gamma$  can be defined as the translation length of a representative of  $[\gamma]$. 
Using ideas developed  in \cite{gK83} and \cite{gK97} we proved in  \cite{CK02}: 
\begin{theorem} \label{thm:conj-Grhyp}
Let $\Gamma$ be a group acting
properly and cocompactly by isometries on
a proper geodesic  Gromov-hyperbolic
metric space  $X$. Assume
that the Gromov boundary $\partial X$ of $X$
contains more than two points.
  For all $t \geq 0$, let
$$
P(t) = \card \{ [\gamma] \mid\; \gamma \in \Gamma \;  \; \text{is primitive and }   \; \; \ell(\gamma) \leq t \},
$$
where $ \gamma \in \Gamma$ is called primitive if it cannot be written as a proper power $\gamma = \alpha^n$ for some $\alpha \in \Gamma$ and $ n \ge 2$.

  Then there exists
constants   $A\ge 1$,  $t_0 >0$ and such that
$$
\frac{1}{A} \frac{ e^{ht} } {t} \leq P(t) \leq A e^{ht}
$$
for all $t \geq t_0$, where $h>0$  is as in \ref{eqn:co93}.
\end{theorem}
\begin{remark} \label{rem:conj-Grhyp}

This results has the following interpretation in the Riemannian setting.
Assume that $(M,g)$ is a closed manifold Riemannian manifold with a Gromov hyperbolic fundamental group
$\pi_1(M)$. Let $\wM$ be the universal covering then $M = \wM/ \Gamma$ where  $\Gamma$   is the group of covering transformations isomorphic to $\pi_1(M)$.
As mentioned above
$h$ is the volume entropy $\hvol(g)$  of $g$ and $P(t)$ is equal to the number of free homotopy classes containing a closed geodesic of period less than $t$.
By a result of Manning \cite{aM79}, the topological entropy $\htop(\phi)$ of the geodesic flow $\phi^t$ on $SM$ is larger than the volume entropy.

\end{remark}

\subsection{Closed manifolds without conjugate points and  Gromov hyperbolic fundamental groups}\label{subsec:Noconjugatepts-Ghyp}
In the following we will assume that $(M,g)$ is a closed Riemannina manifold without conjugate points and Gromov hyperbolic fundamental
group $\pi_1(M)$. As pointed out above in this case Freir\'{e}  and  Ma\~n\'e  \cite{FM82} 
proved that
the volume entropy $\hvol(g)$ is equal to the topological entropy $\htop(\phi)$. 

First we recall the following theorem which is foundational in the theory
of manifolds without conjugate points. 
%(see e.g. \cite{dC, J1, GKM, Kl2, Sa}).
\begin{theorem}(Hadamard-Cartan)\label{thm: Hadamard-Cartan}\\
Let $(M,g)$ be a complete $n$-dimensional Riemannian manifold with
no conjugate points. Then for all $p \in M$ the exponential map
$\exp_p: T_pM \to M$ is a covering map.
In particular complete simply connected manifolds without conjugate points are diffeomorph to
$\R^n$.
\end{theorem}
\begin{remark}
\begin{enumerate}
\item[{\rm (a)}] A complete Riemannian manifold $(M,g)$ has no conjugate points iff for any pair of points on the universal cover $\wM$ there is a unique connecting geodesic geodesics with respect to the lifted metric. In particular all geodesics are minimizing.
\item[{\rm (b)}] The topology of those manifolds is to a large extend determined by the
fundamental group since the contractibility of the universal cover implies that the higher
homotopy groups are vanishing, i.e., $\pi_k(M) = 0$ for $k \ge 2$.
\item[{\rm (c)}] Manifolds of non-positive sectional curvature form an important subclass of manifolds with no conjugate points. Simply connected complete manifolds of nonpositive curvature are called Hadamard manifolds.

%\item[{\rm (b)}] If $M_1, M_2$ are compact manifolds with nonpositive curvature such that
%$\pi_1(M_1)$ and $\pi_1(M_2)$ are isomorphic, then, by a result of Farell and
%Jones, $M_1$ and $M_2$ are homeomorphic. However, in general they are not
%diffeomorphic \cite{FJ1, FJ2}.
\end{enumerate}
\end{remark}

We assume that $(M,g)$ satisfies the divergence property, i.e. 
 for any pair of geodesics  $c_1 \neq c_2$ in $( \wM, g)$ with $c_1(0) =c_2(0)$ 
we have
\begin{equation}\label{eqn:divergence}
\lim\limits_{t \to \infty} d(c_1(t), c_2(t)) = \infty.
\end{equation}
\begin{remark}
\begin{itemize}
\item
By a result of E. Hopf \cite{eH48} all closed non-flat surfaces without conjugate points have genus at least two. Since such surfaces 
carry a metric with negative curvature their fundamental groups are Gromov hyperbolic. Furthermore,  Green \cite{wG56} showed that surfaces without conjugate have the divergence property.
\item
Until now there is no example of a closed manifold without conjugate points known where the divergence property does not hold. A sufficient condition for the divergence property is the continuity of the stable Jacobi-tensors (see \cite{ES76}). In particular, this assumption holds if $(M,g)$ has non-positive curvature, or more generally no focal points.
\end{itemize}
\end{remark}
  The following notion was introduced by Eberlein  in \cite{pEb72} and  Eberlein and  O'Neill \cite{EO73}.
\begin{definition}
A simply connected Riemannian manifold $\wM$ without conjugate points 
is a \emph{(uniform) visibility manifold} if for every $\epsilon>0$ there exists $L>0$ such that whenever a geodesic $c\colon [a, b] \to\wM$ stays at distance at least $L$ from some point $p\in\wM$, then the angle sustained by $c$ at $p$ is less than $\epsilon$, that is
  \begin{equation*}
    \angle_p(c)=\sup_{a\leq s,t\leq b} \angle_p(c(s),c(t))<\epsilon.
  \end{equation*}
\end{definition}
The following Theorem is due to Ruggiero \cite{rR03},
  \begin{theorem}\label{thm:visibility}
Let $(M,g)$ be a closed manifold without conjugate points and Gromov hyperbolic fundamental group. Then $( \wM, g)$ is a visibility manifold if and only if $(M,g)$ has the divergence property.
\end{theorem}
\begin{remark}
Under the stronger assumption that $(M,g)$ is a closed manifold without conjugate points admitting a background metric of negative curvature this has been proved by Eberlein  \cite{pEb72}.
\end{remark}
Since $\wM$ is Gromov hyperbolic,  Lemma \ref{lem:GB} and the divergence property implies that the map
$f_p\colon S_p\wM \to \ideal$ defined by $f_p(v) = c_v(\infty)$ is a bijection.
The topology (sphere-topology) on $\partial \wM$ is defined such that $f_p$
becomes a homeomorphism.
Since for all $q \in \wM$ the map $f_q^{-1} f_p \colon S_p\wM \to S_q\wM$ is a homeomorphism, see \cite{pEb72},
the topology is independent on the reference point $p$.
The topologies on $\partial \wM$ and $\wM$
extend naturally
 to  $\cl (\wM): =  \wM\cup \partial \wM$
by requiring that the map
$\varphi\colon B_1(p) = \{v \in T_p \wM:   \|v\| \le 1\} \to \cl(\wM)$
defined by
\[
\varphi(v) = \begin{cases}
 \exp_p\left(\frac{v}{1-\|v\|}\right) & \|v\| < 1\\
f_p(v) & \|v\| = 1
\end{cases}
\]
is a homeomorphism. This topology, called the cone topology, was introduced by Eberlein
and O'Neill \cite{EO73} in the case of Hadamard manifolds and by Eberlein   \cite{pEb72} in the case of visibility manifolds. In
particular, $\cl( \wM) $ is homeomorphic to a closed ball in
$\mathbb{R}^n$. The relative topology on $\ideal$ coincides with the sphere topology, and the relative topology on $\wM$ coincides with the topology of $\wM$.

Note that for the case which we are considering
this compactification agrees with the compactification of  $\delta$-hyperbolic spaces do to Gromov \cite{mG87}.

For simply connected manifolds $X$ without conjugate points for $v\in SX$, the limit
\[
b_{v}(q) := \lim_{t\to\infty} \left( d(q, c_{v}(t)) - t\right)
\]
exists and is called the \emph{Busemann function} associated to $v$. In \cite{gK85}  it was shown that Busemann functions are of class $C^{1,1}$ provided the sectional curvature is uniformly bounded from below.
\begin{definition}\label{def:busemann}
Let $M$ be a closed manifold without conjugate points, divergence property and Gromov hyperbolic fundamental group. 
Then according to Lemma \ref{lem:GB} and the divergence property, 
for each $p\in X$ and $\xi \in \partial X$ there exists a uniquely determined vector $v\in S_p X$  such that $c_{v}(\infty) = \xi$.
We call $b_\xi(q,p) := b_{v}(q)$ the Busemann function based at $\xi$
and normalized by $b_\xi(p,p) =0$.
\end{definition}

% The relative topology on $\partial  \wM$, respectively, on $ \wM$ coincides with the sphere topology, respectively, the topology on the manifold $ \wM$.
\begin{remark}\label{rem:minmal}
The isometric action of $\Gamma=\pi_1(M)$ on $\wM$ extends to a continuous action on $\ideal$.
Since by \cite{pEb72} the geodesic flow is topologically transitive, every $\Gamma$-orbit in $\ideal$ is dense, i.e. the action on  $\ideal$ is minimal.
\end{remark}

\begin{proposition}\label{cor:busemann}
Let $M$ be a closed manifold without conjugate points, divergence property and Gromov hyperbolic fundamental group. 
Then the following holds
\begin{enumerate}
\item
For $p,q \in \wM$ and $\xi \in \ideal$ we have
$$
\lim\limits_{z \to \xi }d(q,z) -d(p,z) = b_p(q, \xi)
$$
\item
For all $p,q,z\in \wM$ we have
$$b_{q}(z,\xi) = b_p(z,\xi) - b_p(q,\xi)$$
In particular $b_{q}(z,\xi) = -  b_z(q,\xi)$
\end{enumerate}
\end{proposition}

\begin{proof}
For a proof see e.g. \cite{CKW21}
\end{proof}

%\begin{lemma}\label{lem:uniqasym}
%Let $(M,g)$ be a closed Riemannian manifold without conjugate points and of hyperbolic type.
%Then for each $v \in T^1\wM$, we have
%\begin{equation}\label{eqn:Ws-asymptotic}
%W^s(v) =\{w \in T^1\wM : c_w \text{ is asymptotic to }  c_v  \}.
%\end{equation}
%\end{lemma}
%\begin{proof}
%By Lemma \ref{lem:p-to-xi}, both the left- and right-hand sides of \eqref{eqn:Ws-asymptotic} contain exactly one point from each $T_p^1\wM$.  By Lemma \ref{lem:ct}, each $T_p^1\wM$ contains a $w$ that lies in both the left- and right-hand sides.  The result follows.
%\end{proof}

As stated in  Lemma \ref{lem:GB}, for a distinct pair $\xi, \eta \in \ideal $ there exists a in general not uniquely determined geodesic $c$ with $c(-\infty) = \xi$ and  $c(+\infty) = \eta$. In the more restrictive setting of no focal points (in particular, if $M$ has nonpositive curvature), any two distinct geodesics connecting $\eta\neq\xi\in \ideal$ must bound a flat strip in $\wM$, but this is no longer the case in our setting.
Let $\sqbd := \dididi$ and $P: S\wM \to \sqbd$ be the endpoint projection given by $P(v) =(v_-,v_+)$ where $v_- = c_v (-\infty) $ and
$v_+ = c_v (+\infty) $.  Of dynamical importance is the \emph{expansive set}  introduced in the introduction.  Alternatively, it can be defined by the set of flow lines with unique connection of there endpoints, more precisely:
\begin{definition}\label{dfn:expansiveset}
Let $(M,g)$ be a closed manifold without conjugate points and Gromov hyperbolic fundamental group. Then
the set
$$
\mathcal{\widetilde E} = 
\{ v\in S\wM \mid  P^{-1}(P(v)) = \{ \phi_t(v) : t \in \R \} \} 
$$
is called the expansive set on $S\wM$. The projection  of $\wE$ onto $SM$
is called the expansive set on $SM$ and is denoted by  $\EE$. Furthermore we call
$\partial^2 \wE := P(\wE) \subset \sqbd$ the expansive set at infinity.
\end{definition}
\begin{remark}
Note, that $v \in \wE$ is expansive if and only if for all  $w \in \EE$ with
$d(c_v(t), c_w(t)) \le \epsilon$ for some $\epsilon >0$ and all $t \in \R$ there exists $t_0 \in \R$ such that geodesics
$c_v(t) = c_w(t +t_0)$. For  $v \in \EE$  the same statement holds for  $0 <\epsilon \le \inj(M)$ where $\inj(M)$ is the injectivity radius of $M$.

Obviously, the expansive  sets  $\wE$ and $\EE$ are  invariant under the geodesic flow and 
  $\wE$ is invariant under the action of $\Gamma$ on $S \wM$.
Furthermore, the expansive set $\partial^2 \wE := P(\wE)$  at infinity is invariant under the diagonal action of
$\Gamma$ on $\sqbd$ and it consists of all $(\xi, \eta) \in  \sqbd$ with unique connecting geodesic.
\end{remark}
%Now we introduce the following class of geodesic flows on closed manifolds coarse hyperbolic manifolds.
%\begin{definition}\label{dfn: gap}
%Let $(M,g)$ be a closed Riemannian manifold with Gromov hyperbolic fundamental and divergence property.
%We say that the geodesic flow $\phi^t: SM \to SM $ has the gap condition iff
%$$
 %\sup\{h_\mu(\phi) : \mu \in \MMM_\phi(SM), \mu(\EE)=0\} < \htop(\phi)
%$$
%\end{definition}\label{dfn: gap}
%\begin{remark}
%The definition implies that each measure $\mu \in \MMM_\varphi(SM)$  with $h_\mu(\phi) = \htop(\phi)$ has the property that  $\mu(\EE) >0$.
%If $\mu $ is additionally ergodic $\EE$ has full measure.
%\end{remark}
\section{Patterson--Sullivan measure and the MME}\label{sec:PS}
In this section we assume that $M$ is a closed Riemannian manifold without conjugate points having the divergence property of geodesic rays and whose fundamental group is Gromov hyperbolic. Since under those assumptions $\wM$ 
is a visibility manifold it follows as in \cite{CKW21} that the Patterson--Sullivan measure can be used to construct a measure of maximal entropy fully supported on $SM$.  However, to prove ergodicity we need a different argument which does not 
a-priori use the uniqueness of the MME.
%If we add the conditions  \ref{H3} and \ref{H4} from Definition \ref{def:H} we obtain uniqueness 
%as was shown in \S\ref{sec:proof}.

\subsection{Poincar\'{e} series and the Patterson--Sullivan measure}
 
If $\Gamma$ denotes the group of deck transformations, for $p,q \in  \wM$ and $s \in \mathbb{R}$, we consider the Poincar\'{e} series
 \[
P(s,p,q) = \sum\limits_{\gamma \in \Gamma} e^{- sd(p,\gamma q)}.
\]
Since  $\wM$ is Gromov hyperbolic  it follows from \cite {mCo93}  that the series
converges for $s> h$ and diverges for $s \le h$, where $h$ is the topological entropy.
For $x \in \wM$ the set 
$\Lambda(\Gamma)$ of accumulation points of the orbit
 $\Gamma x$ in $ \wM$ is called the limit set. Since $\wM$ is cocompact we have $\Lambda(\Gamma) =
\partial \wM $. 
Fix $x \in \wM $, $s > h$ and consider for each $p \in \wM$  the measure
\begin{equation}\label{eqn:nupxs}
\nu_{p,x,s} = \frac1{P(s,x,x)} \sum\limits_{\gamma \in \Gamma} e^{-sd(p,\gamma x)} \delta_{\gamma x}
\end{equation}
where $\delta_y$ is the Dirac mass associated to $y \in \wM $.
Using the fact that
$e^{-s d(p,x)} e^{-s d(x,\gamma x)} \leq e^{-s d(p,\gamma x)} \leq e^{sd(p,x)} e^{-sd(x,\gamma x)}$ for every $x,p\in \wM$ and $\gamma\in \Gamma$, we see that
\begin{equation}\label{eqn:nu-weight}
e^{-sd(p,x)} \le \nu_{p,x,s} (\cl( \wM) ) \le e^{sd(p,x)};
\end{equation}
in particular, the $\nu_{p,x,s}$ are all finite.  Moreover, we clearly have
\begin{equation}\label{eqn:supp-nu}
\Gamma x \subset \supp \nu_{p,x,s} \subset \overline{\Gamma
x}.
\end{equation}
Now choose for a fixed $p \in \wM  $ and a weak limit $\lim\limits_{k\to \infty} \nu_{p,x,s_k} = : \nu_p$.

The divergence of the series $P(s,x,x)$ for $s = h$ and the
discreteness of $\Gamma$ yields that the support of
$\nu_{p}$ is contained in the limit set. Moreover, one obtains:

\begin{proposition}\label{prop:PS} There is a sequence \(s_k\to h\) as \(k\to\infty\) such that 
for every $p \in \wM$ the weak* limit $\lim\limits_{k\to \infty} \nu_{p,x,s_k} =:
\nu_p$ exists.
The  family of measures $\{\nu_p\}_{p \in \wM}$ has the following 
properties.
\begin{enumerate}[label=\upshape{(\alph{*})}]
\item\label{PS-a} $\{\nu_p\}_{p \in \wM}$ is $\Gamma$-equivariant: for all Borel sets $A \subset \partial \wM$, we have
\[
\nu_{\gamma p} (\gamma A) = \nu_p(A).
\]
\item\label{PS-c} $\frac{d \nu_q}{d\nu_p} (\xi) = e^{-h b_p(q, \xi)}$
for almost all $\xi  \in \partial \wM$, where $b_p(q,\xi)$ is as in Definition \ref{def:busemann}.
\item\label{PS-b} $\supp \ \nu_p = \partial \wM$ for all $p \in
\wM$.
\end{enumerate}
\end{proposition}
\begin{proof} 
For a proof see e.g. \cite{CKW21}.
\end{proof} 
For $\xi \in \partial  \wM$ and $x \in \wM$
consider the projections
\[
pr_\xi \colon  \wM \to \partial  \wM
\quad\text{and}\quad
pr_x \colon  \wM \setminus \{x\} \to  \partial  \wM
\]
along geodesics emanating from $\xi$ and $x$, respectively. That is, $pr_\xi(q)
= c_{\xi, q}(\infty)$, where $c_{\xi,q}$ is the geodesic with
$c_{\xi,q}(-\infty) = \xi$,
 $c_{\xi, q}(0) = q$ and $pr_x(q) =
c_{q,x}(\infty)$, where $c_{q,x}$ is the geodesic
 with $c_{q,x}(0) = q$,
$c_{q,x}(-d(q,x))=x$. 
\begin{lemma}\label{lem:open-shadow}
Consider  $R =1 +2\delta$ where $\delta$ is  the Gromov hyperbolicity constant for  $\wM$. Then 
for all $p \in \wM$ and $x \in  \cl(\wM) \setminus B(p, R) $, the \textit{shadow} 
$pr_x B(p,R)$  of the open geodesic ball $B(p,R)$ with center $p$ and radius $R$ contains an open set in $\partial  \wM$.
\end{lemma}

\begin{proof} 

For \(x\in\cl{(\wM)} \setminus B(p,R)\) and \(p\in\wM\), let $c_0: \R \to \wM$ denote the unique geodesic with $c_0(0) =p$  and $c_0(-d(x,p)) =x$ if $x \in \wM$ and $c_0(-\infty) =x$ if $x \in \partial \wM$. Define 
$v := \dot c_0(0)$. By the definition of the topology on \(\partial \wM\),  for every  \(v\in S_p\wM\) and \(\varepsilon>0\) we have \(A_{\varepsilon}(v):=\{c_w(\infty): \angle_p(v,w)<\varepsilon\}\) is open in \(\partial\wM\). 
The continuous dependence of geodesics on initial conditions implies that we can choose $\varepsilon>0$ such that
$d(c_v( \pm (\delta +1) , c_w( \pm (\delta +1)) \le 1$ for all $ w \in S_p \wM$ with $ \angle_p(v,w)<\varepsilon $.
 For each $\eta =c_w(\infty)  \in A_{\varepsilon}(v)$ it follows from Corollary 4.5 in \cite{gK12}
 that for all $t >0$  and geodesics  $c_{t, w}$  with $c_{t, w}(T) = c_w(t)$ and  $c_{t,w}(0) =x$  for $x \in \wM$ resp. 
 $c_{t,w}(-\infty) =x$  for $x \in \partial \wM$  there exists $s \in [0,T]$  for $x \in \wM$  resp.  $s \in (-\infty ,T]$ for $x \in \partial \wM$  
with $c_{t,w} (s) \in B(p,R)$. Choose a sequence a sequence $t_n \to \infty$ such that the geodesics $c_{t_n,w}$ converge to a geodesic $c$. Since $c_{t_n,w} (T_n) =  c_w(t_n) \to \eta$ and  $c_{t_n,w}(s_n) \in   B(p,R)$ we have  $c (\infty)  = \eta $ and  there exists some $s_0 \in \R$ with  $c(s_0) \in B(p,R)$. 
This implies that \(A_{\varepsilon}(v)\subset pr_x B(p,R)\).

\end{proof}
Lemma \ref{lem:open-shadow} yields:

\begin{proposition}\label{prop:PS-bds}
Let $\{\nu_p\}_{p\in  \wM }$ be the Patterson--Sullivan measures and fix $\rho\geq R$, where $R$ is as in Lemma \ref{lem:open-shadow}.
\begin{enumerate}[label=\upshape{(\alph{*})}]
\item\label{bds-a}
There exists $\ell=\ell(\rho) > 0$ such that for every $x\in \cl(\wM) \setminus  B(p,\rho)$, we have
\[
\nu_p(pr_x B(p,\rho)) \ge \ell.
\]
\item\label{bds-b}
There is a constant $b = b(\rho)$
such that for all $x \in \wM$ and $\xi = c_{p,x}(-\infty)$,
\[
\frac{1}{b} e^{-h d(p,x)} \le \nu_p(pr_\xi(B(x,\rho)) \le be^{- h d(p,x)}.
\]
\item\label{bds-c}
A similar estimate holds if we project from $p \in \wM$, namely there is a constant
$a = a(\rho) > 0$ such that for all $p\in \wM$,
\[
\frac{1}{a} e^{-h d(p,x)} \le \nu_p(pr_p(B(x,\rho))) \le a e^{-h
d(p,x)}.
\]
\end{enumerate}
\end{proposition}
\begin{proof}
For a proof see \cite{CKW21}.
 \end{proof}
 
  \subsection{Construction of the measure of maximal entropy} \label{sec: PS measure}
Now we construct an
invariant measure for the geodesic flow using the Patterson-Sullivan measures $\nu_p$. We follow the approach in \cite{CKW21}, which was originally carried out in negative curvature in \cite{vK90} and in \cite{gK98} for rank 1 manifolds. However, as was explained in \cite{CKW21}, the construction in the present setting introduces some technical difficulties 
which we managed to overcome. We refer to \cite{CKW21} for details.

By Proposition \ref{prop:PS}\ref{PS-c},  $\nu_p$ is
 $\Gamma$-quasi-invariant with Radon-Nikodym cocycle
\begin{equation}\label{eqn:RN}
f(\gamma, \xi) = e^{-h b_p(\gamma^{-1} p, \xi)} = \frac{d\nu_{\gamma^{-1} p}}{d\nu_p}(\xi).
\end{equation}
For $(\xi, \eta) \in \sqbd := \dididi$ consider
\begin{equation}\label{eqn:beta-p}
\beta_p(\xi, \eta) = - (b_p(q, \xi) + b_p(q, \eta)) \;,
\end{equation}
where $q$ is a point on a geodesic $c$ connecting $\xi$ and $\eta$.
In geometrical terms $\beta_p(\xi, \eta)$ is the length of the
segment $c$ which is cut out by the horoballs through
$(p, \xi)$ and $(p, \eta)$.
Since $\grad_q b_p(q,\xi) =- \grad_q b_p(q,\eta)$ for all points on geodesics
connecting $\xi$ and $\eta$, this number is independent of the choice of
$q$. An easy computation using \eqref{eqn:RN}, see \cite[Lemma 2.4]{gK98}, shows:
\begin{proposition}\label{prop: measure barmu}
For $p \in \wM$, the measure $\bar\mu$ on $\sqbd$ defined by
\[
d \bar\mu (\xi, \eta) = e^{h\beta_p(\xi,\eta)}
 d\nu_p(\xi) d\nu_p(\eta)
\]
is invariant under the diagonal action of $\Gamma$ and ergodic. 
\end{proposition}
\begin{proof}
The proof of the $\Gamma$-invariance follows from \eqref{eqn:RN}, see \cite[Lemma 2.4]{gK98}.
The ergodicity is proved in \cite[Theorem 1.4]{BF17}. See also \cite{CDST25}  for a proof in a more general context.
\end{proof}

\begin{theorem}\label{thm:measure mu}
Let $(M,g)$ be a smooth closed Riemannian manifold without conjugate points with divergence property and Gromov hyperbolic fundamental group. Then there exists a measure of maximal entropy $\mu$ for the geodesic flow on $SM$ whose lift $\tilde \mu$ to
$S\wM$ is of the form
\[
\tilde\mu(A) = \int_{\sqbd} 
\mu_{\xi,\eta}(A) \,d\bar\mu(\xi,\eta).
\]
where  $\mu_{\xi,\eta}$ is supported on  $P^{-1}(\xi,\eta)$. 

Assume additionally that either for all non-expansive measures 
$\nu\in \MMM_\phi(SM)$ (i.e. $\nu(\EE) =0 $)  we have $h_\nu(\phi) < \htop( \phi)$
 or the expansive  set $\EE$ has non-trivial interior in $SM$.
Then $\partial ^2\wE$ has full measure w.r.t. $\bar \mu$ and therefore
$P^{-1}(\xi,\eta)$ is a single geodesic for $\bar\mu$-a.e.\ $(\xi,\eta) \in \sqbd$. Furthermore, for all measurable sets $A \subset S\wM$
we have
$$
\mu_{\xi,\eta}(A) =  \Leb \{ t \in \R : \phi_t(v) \in A \}.
 $$
 for $v \in P^{-1}(\xi,\eta)$, and hence $\mu(\EE) = 1$. Furthermore, the corresponding normalized measure $\mu$ on $SM$ is mixing, has full support in $SM$.
\end{theorem}
\begin{proof}
For the construction of the measure $\mu$  and the proof that $\mu$ is a measure of maximal entropy we refer to \cite{CKW21}.

Assume that  $h_\nu(\phi) < \htop( \phi)$ for all non-expansive measures 
$\nu\in \MMM_\phi(SM)$. 
Therefore, $h_\mu (\phi) = \htop (\phi)$ yields $\mu( \EE) >0$ and 
by the construction of $\mu$ this implies $\bar\mu (\partial^2\wE)$  is positive. If alternatively $\EE$ in $SM$ has non-trivial interior then by the definition of the topology in $\partial \wM$  the set $\partial^2\wE$ has  non-trivial interior in $\sqbd$ as well. This also implies  $\bar\mu (\partial^2\wE)$  is positive since by proposition \ref{prop:PS} $\bar\mu$ has full support. 
Hence, in both cases the ergodicity of   $\bar\mu $ and the  $\Gamma$-invariance of $\partial^2 \wE$ implies
 $\bar\mu (\wE) = \bar\mu(\sqbd)$.
  In particular, we have $\mu (\EE) =1$ for the probability measure $\mu$
  and  $\mu_{\xi,\eta}(A) =  \Leb \{ t \in \R : \phi^t(v) \in A \}$ for all measurable sets $A\subset S\wM$
and $\bar\mu$ almost all  $(\xi, \eta) \in \sqbd$.

   To prove ergodicity of $\mu$ we consider a measurable $\phi^t$ invariant  set $A \subset SM$ with $\mu(A) >0$. Since  
  $\mu (\EE) =1$ we can assume 
  that $A$ is contained in $\EE$.  Therefore, its lift $\widetilde A$  to the cover   $S\wM$ is contained in $ \wE$. Since  $\widetilde A$ is $\phi^t$ invariant and  $P^{-1}(\xi,\eta)$ for $(\xi, \eta) \in \partial^2 \wE$ consists of a single geodesic there is a $\Gamma$-invariant subset $B \subset  \partial^2\wE$ with $P^{-1} B = \widetilde A$. From $\mu(A) >0$ we conclude $\bar \mu (B) >0$ and the ergodicity  of $\bar \mu$ implies $\bar \mu (B) = \bar \mu (\sqbd)$. Hence
 $\mu(A) = 1$ and therefore the ergodicity of $\mu$ follows. 
 
 Since $\bar \mu$ has full support in $\sqbd$, $\mu$ has full support in $SM$ as well. Furthermore, using in argument of Babillot in \cite{mB02} we proved in \cite{CKW21} that the ergodicity of $\mu$ even implies that $\mu$ is mixing.

   \end{proof}
   %%%%%%%%%%%%%%%%%%%%%%%%%%%%%%%%%%%%%%
   %%begincomment
 \section{Entropy expansiveness}\label{sec:e-expansiveness}
Let $V$ be a compact metric space, $T: V \to V$ a homeomorphism and $F 
\subset V$ be a set in $V$. A set $E = E (n, \delta, F) \subset V$ is called an 
$(n, \delta)$-span of $F$ (with respect to $T$) if for each $y \in F$ there 
exists an $x \in E$ such that $d (T^k x,T^k y) \le \delta$ for all $0 \le k < 
n$. Let 
$$
r_n (F, \delta) =\min \{\card E \mid E {\rm \ is \ an \ } (n,\delta) {\rm -span\ 
of\ } F\}.
$$
Since $V$ is compact, $r_n(F,\delta) < \infty$. Define
$$
r(F,\delta) = \limsup_{n} \frac{1}{n} \log r_n(F,\delta)
$$
and 
$$
h(T,F) = \lim_{\delta \to 0} r(F,\delta).
$$
Note, that $h(T,F)$ is independent of the metric generating a given 
topology. $h(T) = h(T,V)$ is called the topological entropy of $T$. Now we 
recall the notion of entropy expansiveness introduced by R.~Bowen \cite{rBo72}.
\begin{definition}\label{def:Bowen-expansive}   
For $x \in V$ and $\epsilon>0$ consider the set
$$
Z_\epsilon(x) = \bigcap\limits_{n \in \Z} T^{-n} B(T^n x, \epsilon) = \{y\in X 
\mid d(T^nx, T^n y) \le \epsilon  \; \; \text{for all} \; \; n \in \Z \},
$$
where $B(y, \epsilon)$ denotes the $\epsilon$-ball about $y$ with respect to $d$.
A homeomorphism $T: V \to V$ is called entropy expansive at scale
$\epsilon > 0$ if $ h(T,Z_\epsilon(x)) = 0$ for all $x \in V$.
\end{definition}

\begin{remark}
Note that if $T$ is an
expansive homeomorphisms  at scale $\epsilon> 0$, i.e.  if $Z_\epsilon(x) = \{x\}$ for all $x\in V$
then it is  trivially entropy expansive at the same scale.
\end{remark}

We recall the following result of Bowen for measure preserving 
homeomorphisms on compact metric spaces $V$.

\begin{theorem}\label{thm:Bowen} (Bowen \cite{rBo72})    %{3.2} (Bowen)\\
Let $V$ be a finite dimensional compact metric space, $T: V \to V$ a homeomorphism and
${\AAA} = 
\{A_1, \ldots, A_m\}$ a Borel partition such that 
$$
\diam {\AAA} :=\max \{\diam A_i\} < \epsilon.
$$ 
Let $\nu$ be a $T$-invariant Borel probability measure. Then 
$$
h_\nu (T) \le h_\nu(T, {\AAA}) + \sup\limits_{x\in V} h(T,Z_\epsilon(x)).
$$
\end{theorem}
\begin{remark}
In particular \ref{thm:Bowen} yields: if $T$ is entropy expansive at scale $\epsilon$, 
 and ${\AAA}$ a partition  of $V$ with $\diam {\AAA} < \epsilon$, then
$$
h_\nu(T) = h_\nu(T, {\AAA})
$$
holds.
\end{remark}
A very useful generalization of entropy expansiveness has been introduced by Climenhaga and Thompson
\cite{CT16}
\begin{definition}\label{def:CT-expansive} 
Let $\nu$ be a $T$-invariant Borel probability measure of a homeomorphism $T: V \to V$. Then $T$ is called $\nu$  almost entropy expansive at scale 
$\epsilon > 0$ if $ h(T,Z_\epsilon(x))= 0$ for $\nu$ almost $x \in V$.
\end{definition}

%\begin{remark}
%If $T$ is $h$-expansive the map $\nu \to h_\nu(T)$ is upper semi-continuous 
%on  $\MMM(T)$ i.e.
%$\limsup_{n \to \infty} h_{\nu_n} (\phi)\le h_\nu(\phi)$ if
% $\nu_n \to \nu$  in the weak* topology. 
%To see this one chooses
%for $\nu \in \MMM(T)$ a partition $\AAA = \{A_1, \ldots, A_m\}$ such that $\diam 
%\AAA$ is smaller than the
%expansivity constant and $\nu(\partial A_i) = 0$ for all $A_i \in \AAA$.
%Then  $h_{\nu_n}(T) = h_{\nu_n}(T, {\AAA}) \le  
%\frac{H_{\nu_n}(\AAA_T^{(k)})}{k}$  and
%$\lim_{n \to \infty}H_{\nu_n}(\AAA_T^{(k)})  = H_{\nu}(\AAA_T^{(k)})$ for all $k \in \N$. Therefore, the  upper 
%semi-continuity of entropy follows from its  definition.
%\end{remark}
Using this notion they obtain in \cite{CT16}, Theorem 3.2.
\begin{theorem}\label{thm:CT} 
Let $V$ be a compact metric space, $T: V \to V$ a homeomorphism and
${\AAA}$ a partition with $\diam {\AAA} <\epsilon$. If  $\nu$ is an ergodic $T$-invariant Borel probability measure
such that $\nu$  is almost entropy expansive at scale 
$\epsilon > 0$. Then 
$$
h_\nu(T) = h_\nu(T, {\AAA}). 
$$
\end{theorem}
For a given Riemannian manifold $(M,g)$ we define 
for $k \in \N$ the metric on $SM$ given by
\begin{equation}\label{eqn:dyn-metric}
d_k(v,w) = \max \{d(c_v(t), c_w(t) \mid t \in [0,k] \},
\end{equation}
where $d$ is the metric on $M$ induced by the Riemannian metric $g$.
Using Theorem \ref{thm:CT} we obtain:
\begin{proposition}\label{prop:entexpansive}  
Let $(M,g)$ be a compact Riemannian manifold  without conjugate points and $\phi^t$ the 
geodesic flow on the unit tangent bundle $SM$. Assume that $\nu$ is a $\phi^t$- invariant probability measure such that  $\nu(\EE) =1$,  where $\EE$ is the expansive set introduced in section \ref{sec: PS measure}.
 If $\epsilon > 0$ is smaller then the injectivity radius $\inj(M)$ of $M$, then $\phi^k$, $k \in \N$,  is $\nu$ almost everywhere entropy expansive with respect to the metric $d_k$.
  If additionally $\nu$ is ergodic  and   ${\AAA}$ is a Borel partitions with $\diam_{d_k} {\AAA} <
\inj (M)$,

$$
h_\nu(\phi^k) = h_\nu(\phi^k, {\AAA}).
$$

\end{proposition}
\begin{proof}
For $v \in \EE$ be a vector in the expansive set. Then
\begin{align*}
Z_{\epsilon} (v) &= \{w\in SM \mid d_k (\phi^{nk}w, \phi^{nk}v) \le \epsilon, n \in 
\Z \} \\
&= \{w \in SM \mid d(c_v(t), c_w(t)) \le \epsilon, t \in \R \}.
\end{align*}

Consider $w  \in Z_{\epsilon} (v)$ and the function $ L(t) = 
d(c_w(t), c_{v}(t)) \le \epsilon \le \inj(M)$. Therefore there exists lifts
$\wti  c_{v}, \wti  c_{w}: \R \to \wM$ such that
$$
d(\wti  c_v(t), \wti  c_w(t)) \le \epsilon
$$
fot all $ t \in \R$. Since $v \in \EE$ we have that the images of geodesics $\wti  c_v$ and  $\wti  c_v$ agree.
In particular, $ Z_{\epsilon} (v) =  \{ \phi_t (v) \mid t \in \R \}$.

If additionally $\nu$ is ergodic  and   ${\AAA}$ is a Borel partitions with $\diam_{d_k} {\AAA} <
\inj (M)$, then Theorem \ref{thm:CT} implies 
$
h_\nu(\phi^k) = h_\nu(\phi^k, {\AAA}).
$

\end{proof}
%\begin{remark}
%Obviously, the geodesic flow on nonpositively curved manifolds is in general 
%not expansive.
%\end{remark}

\section{The uniqueness of the measure of maximal entropy}  

In this section we want to proof the uniqueness of the MME under the conditions stated in Theorem  \ref{thm:main} and
 Theorem  \ref{thm:expansiveMME}. 
 We assume that either for all non-expansive measures 
$\nu\in \MMM_\phi(SM)$ (i.e. $\nu(\EE) =0 $)  we have $h_\nu(\phi) <\htop(\phi)$
 or the expansive  set $\EE$ has non-trivial interior in $SM$.
According to Theorem \ref{thm:measure mu} this implies
that the measure constructed in subsection \ref{sec: PS measure}
is ergodic of maximal entropy and $\mu(\EE) = 1$.
Moreover, $\mu$  is the normalized projection of the measure $\tmu$ on $S \wM$ with  
$$
\tmu(A) =  \int_{\sqbd} \Leb \{ t \in \R : \phi_t(v) \in A \}d \bar\mu (v_- , v_+)
$$
where $v_\pm = c_v(\pm \infty)$ and
$\bar\mu$ is  defined in Proposition 
\ref{prop: measure barmu} 
using the Patterson-Sullivan measure $\nu_p$ on $\partial \wM$ by
\[
d \bar\mu (\xi, \eta) = e^{h\beta_p(\xi,\eta)}
 d\nu_p(\xi) d\nu_p(\eta)
\]

 In  this section we prove first that $\mu$  is the unique MME under the following conditions.
\begin{theorem}\label{thm:uniqueMME-inj}
Let $(M,g)$ be a closed Riemannian manifold without conjugate points with divergence property and with Gromov hyperbolic fundamental group. Let $\delta $ the Gromov hyperbolicity constant of the universal cover $\wM$ and assume that the  injectivity radius  $\inj(M)$ of $M$ is larger than $16 \delta$. Assume that at least one of the following two
conditions is fulfilled:
\begin{enumerate}
\item 
The entropy of non-expansive measures is strictly smaller than the topological entropy, i.e.
$$ h_\nu(\phi)< \htop(\phi) \; \; \text{ for all} \; \;  \nu \in \MMM_\phi(SM) \; \; with \; \;\nu(\EE)=0$$ 
\item
The expansive set $\EE \subset SM$ has non-empty interior and the geodesic flow $\phi^t: SM \to SM$ is entropy expansive
at some scale larger than  $8 \delta$. 
\end{enumerate}
Then the measure $\mu$ is the unique measure of maximal entropy of the geodesic flow.
\end{theorem}

We first establish properties of $\mu$ which will be of importance in the proof of Theorem \ref{thm:uniqueMME-inj}.
 
\subsection{Properties of the measure $\mu$ and partitions}
Let $\delta $ be the Gromov hyperbolicity constant of the universal cover $\wM$ and  $R=1 + 2\delta$.
Let ${\FFF}$ be a fundamental domain 
containing
a reference point $p \in \wM$ and $R' \ge q(R)+ 1/2$, where $q(R) = 19 \delta + 4 R$
such that ${\FFF} \subset B(p, R')$. If 
$A \subset \wM$  let $SA$ be the set of all unit 
vectors with footpoint in $A$. For $x \in \wM$ and $R > 0$ 
let
$$
D(x,R', R) = \{v \in SB(p, R') \mid c_v(r) \in B(x, R)  
\mbox{\rm \ for\ some\ } r>0\}.
$$
Then the following lemma holds.
\begin{lemma}\label{lem: lower bound mu(D)}     
Consider $x \in X$ such that $d(p,x) >  R$. If $R' \ge  q(R) + \frac {1}{2}$ where $q(R) = 19 \delta + 4 R$  then 
$$
\tmu(D(x,R', R)) \ge c' e^{-h d(p,x)} 
$$
 for a constant $c' > 0$ which does not depend on $x$.
\end{lemma}
\begin{proof}
Consider the set $pr_x(B(p,R))$.
According to Proposition  \ref{prop:PS-bds} there is a constant $\ell > 0$ such that 
$\mu_p(pr_x(B(p,R))) \ge \ell $. If $\eta \in pr_x(B(p,R))$ then  $c_{\eta, x} (-r) \in B(p,R)$
for some $r >0$ where $c_{\eta, x}$ is the geodesic with $c_{\eta, x}( -\infty) = \eta$ and  $c_{\eta, x}(0) =x$.
If  $c$ is a
geodesic with $c(-\infty) =\eta$ and $c (0) \in B(x,R)$ then Lemma \ref{lem:GB} implies
$$
d(p, c(-r)) \le d(p,c_{\eta, x} (-r)) + d(c_{\eta, x} (-r), c(-r)) \le R + 19 \delta + 3R = q(R)
$$
Hence  $c(0) \in B(x, R)$ implies $v = \dot c (-r) \in D(x,R', R)$ and if 
$P: S\wM \to \sqbd$ is the endpoint 
projection map given by $P(v) = (c_v(-\infty), c_v (+\infty))$ we obtain that $P(D(x,R',R))$
contains the set 
$$
A := \bigcup_{\eta \in pr_x(B(p,R))} \{\eta\} \times pr_\eta (B(x,R)).
$$
For $\eta \in pr_x(B(p, R))$ choose a point $q \in  B(p, R)$
that lies on the geodesic $c_{\eta,x}$.
Then the definition of the Patterson-Sullivan measure and Proposition \ref{prop:PS-bds} 
imply
$$
\mu_p(pr_\eta(B(x, R))) \ge e^{-h R} \ \mu_q(pr_\eta(B(x, R)))
\ge \frac{e^{-h R}}{b} e^{-h d(q, x)} \ge b' e^{-h d(p, x)}
$$
for a constant $b' >0 $. Since for a given  point in $A$ the time a corresponding
geodesic  spends in
$D(x,R',R)$ is larger than $2(R' -q(R)) \ge 1$, the desired estimate follows
from the definition of the measure $\tmu$.
\end{proof}

\begin{lemma}\label{lem: r_0 separated}   
Choose some $\rho >0$  and assume $x \in \wM$ such that $d(x, p) \ge n +R+ R'$. Then  for $r_0 = 4 \delta + 3\rho$
the  cardinality of a $(d_n,2 r_0) $-separated set of $D(x,R', R)$ is  
bounded from above by a constant $a( \rho ,R', R)$. 
\end{lemma}
\begin{proof}
Choose maximal $\rho$-separated sets $P= \{p_1 \ldots p_m\}$ of
$B(p,R')$ and 
$Q= \{q_1 \ldots q_\ell\}$ of $B(x,R)$. Their cardinalities are obviously 
bounded by numbers depending only on $R'$ and $R$. Let $c_{ij}$ denote 
the geodesic with $c_{ij}(0) = p_i$ and $c_{ij}(d(p_i,q_j))= q_j$. 
Therefore, if 
$c: [0,T] \to X$ is a geodesic arc connecting  points in $B(p,R')$ and 
$B(x,R)$. Then $T \ge n$ and there exist $p_i \in B(p,R')$ and $q_j \in B(x,R)$ such that 
$c(0) \in 
B(p_i, \rho)$ and $c(T) \in B(q_j, \rho)$. 
By Lemma \ref{lem:endpts-suffice}
 we have $d(c(t), c_{ij}(t)) \le 4 \delta + 3\rho =r_0$ 
for $0 \le t \le n$. If we define $v_{ij} = \dot c_{ij}(0)$ it follows from 
the definition \ref{eqn:dyn-metric} of $d_n$ that $\cup B_{d_n}(v_{ij}, r_0) \supset D(x,R',R)$. 
Hence, the cardinality of a $(d_n, 2r_0)$-separated set of $D(x,R',R)$ is bounded by
$\ell \cdot m$. The number $\ell \cdot m$ is obviously bounded by a constant 
depending only on $\rho$,  $R$ and $R'$. 
\end{proof}
Let $ \FFF$ be a fundamental domain with $p\in \FFF$ and  $R' \ge q(R) + \frac {1}{2}$ such that $\FFF \subset B(p,R^\prime)$ as above. Choose $r(n) = 2n + R' + 2R$. 
Consider a maximal 
$2R$-separated set $x_1 \ldots x_{k(n)}$ of the geodesic sphere
$S(p,r(n))$ of radius $r(n)$ about $p$. Since this set is $2R$-separated,
the balls $B(x_i,R)$ are pairwise disjoint, and since it is maximal, the
balls $B(x_i,2R)$ cover $S(p,r(n))$. Consequently the sets $D(x_i, R', R)$
are pairwise disjoint and the sets $D(x_i, R',2R)$ cover $SB(p,R')$. Thus
we can choose a partition 
$F_{1}^{2n} \ldots F_{k(n)}^{2n}$ of $SB(p,R')$ such that 
$$
D(x_i, R', R) \subset F_{i}^{2n} \subset D (x_i, R', 2R)
$$
for each $i$. Let $Q:SB(p,R') \to SM$ be the restriction to $SB(p,R')$ of the
projection from  $S\wM$ to $SM$ and let $L_i^{2n} = Q(F_i^{2n})$.
\begin{lemma}\label{lem:coveringL}      
$L^{2n} = \{L^{2n}_1 \ldots L^{2n}_{k(n)}\}$ is  a covering of $SM$.
There are constants $\alpha, \beta > 0$ independent of $n$, such that any $v \in SM$
lies in at most $\alpha$ sets from $L^{2n}$ and $\mu(L_i^{2n}) \ge \beta e^{-2hn}$ 
for all $i$.
\end{lemma}
\begin{proof}
Let $\alpha$ be the cardinality of elements $\gamma \in \Gamma$ such that $\gamma(\FFF) \cap B(p, R') \not= 0$.
Then the map $Q$ is at most $\alpha$ to 1. Also $Q$ is onto, because $B(p, R') \supset
\FFF$. Surjections that are at most $\alpha$ to 1 carry partitions to covers that are at most
$\alpha$ to 1. \\
We also have $\alpha \mu(Q(A)) \ge \mu(A)$ for any measurable set $A \subset SB(p,R')$, in 
particular $A = D(x_i, R',R)$. Since $F_i^{2n} \supset D(x_i, R',R)$, it follows
using Lemma \ref{lem: lower bound mu(D)}   that $\mu(L_i^{2n}) \ge c' \alpha^{-1}e^{-2hn}$, where $c'$ is as in Lemma 
 \ref{lem: lower bound mu(D)}  
\end{proof}
Let $ V = \{v_1 \ldots v_m \}$ be a maximal
($d_{2n}$, $2r_0$)-separated set of $SM$ where $r_0 =4 \delta + 3\rho$ as in Lemma \ref{lem: r_0 separated}. Choose a 
partition $\BBB$ such that for each $B \in \BBB $ there exists 
$v_i \in V$ such that 
$$
B_{d_{2n}}(v_i,r_0) \subset B \subset B_{d_{2n}} (v_i, 2r_0).
$$
\begin{lemma}\label{lem: est-partition}     
Let $L^{2n}= \{L_i^{2n}\}$ be the covering introduced above. Then for each 
fixed i 
$$
\card \{B \in \BBB \mid B \cap L_i^{2n}\not= \emptyset\} \le a(\rho, R',2R),
$$
where $a( \rho,  R', 2R)$ is as in Lemma \ref{lem: r_0 separated}.   

\end{lemma}
\begin{proof}
From the definition of $\BBB$ follows that 
$$
\card\{B \in \BBB \mid B \cap L_i^{2n} \not= \emptyset\}
\le \card\{v_j \in V \mid B_{d_{2n}}(v_j, 2r_0)  \cap L_i^{2n} 
\not= \emptyset\} .
$$ 
Since $L_i^{2n} = Q (F_{i}^{2n})$ we can lift the set   
$\{v_j \in V \mid B_{d_{2n}}(v_j, 2r_0)  \cap L_i^{2n} 
\not= \emptyset\}$ to an $2r_0$-separated set of
$F_{i}^{2n} \subset D(x_i,R',2R)$. Therefore the result follows from
 Lemma \ref{lem: r_0 separated}.
\end{proof}
\begin{lemma}\label{lem:sep-measures} 
Let $\FFF \subset \wM$ be  a fundamental domain, $\bar \FFF$ its closure and
let $\nu$ be a Borel probability measure on $SM$. For a fixed 
constant $b > 0$ consider a sequence $A^n$ of measurable
coverings of $SM$ such that for 
each $n$ and $v,w\in A \in A^n$ there are lifts $\tilde v,\tilde w \in S\bar \FFF$ such that $d(c_{\tilde v}(t),(c_{\tilde w}(t)) \le b$
for all $t \in [-n,n]$. Let
$\Omega \not= SM$ be a set containing the complement $SM \setminus \EE$ of the expansive set $\EE$.  
Then there exists a union $C_n$ of subsets of $A^n$ such 
that $\nu(C_{n_j} \Delta \Omega):= \nu(C_{n_j} \setminus \Omega)+ \nu(\Omega \setminus C_{n_j}) 
\to 0$ for some subsequence $n_j$ as $j \to \infty$.
\end{lemma}
\begin{proof}
For each $ j \in \N$ we can choose compact sets $K_1 \subset \Omega, K_2 \subset SM \setminus \Omega$  
with $\nu(\Omega \setminus K_1) < \frac{1}{j} $ and $\nu((SM\setminus \Omega)\setminus 
K_2) < \frac{1}{j}$. 
Consider $C_n:= \bigcup \{A\in A^n \mid A \cap K_1 \not= \emptyset\}$. 
Then there exists $n_j \in \N$ such that for all $v\in K_1$ and $w\in K_2$ and all lifts
$\tilde v, \tilde w \in S \bar \FFF$ such that  $d(c_{\tilde v}(t),c_{\tilde w}(t)) > b$ for some
$t \in[-n_j,n_j]$. 
Suppose this is not 
the case, we would find for all $n \in \N$ elements
$v_n \in K_1$,  $w_n \in K_2$ and lifts $\tilde v_n \in  S \bar \FFF $, $\tilde w_n \in  S \bar \FFF$ 
such that  $d(c_{\tilde v_n}(t),(c_{\tilde w_n}(t)) \le b$ for all
$t \in[-n,n]$.  By the compactness of $K_1, K_2$ and   $S \bar \FFF $ we can assume that $v_n$,  $w_n $, 
 $\tilde v_n$ and $\tilde w_n \in  S \bar \FFF$ converges to $v_0 \in K_1$, $w_0 \in K_2$, $ \tilde v_0 \in  S \bar \FFF$ and $w_0 \in  S \bar \FFF$. This implies that  $d(c_{\tilde v_0}(t),(c_{\tilde w_0}(t)) \le b$ for all
$t \in \R$ which contradicts the fact that $w_0$ belongs to the expansive set $\EE$.  

The construction of $C_n$ implies $K_1 \subset C_n$. Furthermore,  we have 
$C_{n_j} \cap K_2 = \emptyset$, since otherwise we would find $A \in A^{n_j}$ with 
$w \in K_2 \cap A$ and $v \in A \cap K_1$. The definition of  $A^{n_j}$ implies that
there are lifts $\tilde v,\tilde w \in S\bar \FFF$ such that $d(c_{\tilde v}(t),(c_{\tilde w}(t)) \le b$
for all $t \in [-n_j,n_j]$, in contradiction to $d(c_{\tilde v}(t),c_{\tilde w}(t)) > b$ for some
$t \in [-n_j,n_j]$. 
Therefore, we obtain:
$$
\nu(C_{n_j} \Delta \Omega) = \nu(C_{n_j} \setminus \Omega) + \nu(\Omega \setminus C_{n_j}) \le 
\nu((C_{n_j} \setminus \Omega) \setminus K_2) + \nu(\Omega \setminus K_1) \le \frac{2}{j}.
$$
\end{proof}
\begin{lemma}\label{4L}   
Let $L^{2n} =\{L_1^{2n} \ldots L_k^{2n}\}$ be the covering of $SM$ introduced
in Lemma \ref{lem:coveringL}.
Then the 
sequence of coverings $A^n = \{\phi^n L_i^{2n}\}$ has the property of 
$A^n$ stated in Lemma \ref{lem:sep-measures}.  
\end{lemma}
\begin{proof}
Let $Q:SB(p,R') \to SM$ be the restriction to $SB(p,R')$ of the
projection from  $S\wM$ to $SM$.  Then for each  $i \in \{1, \ldots  k$  we have $L_i^{2n} = Q(F_i^{2n})$.
Since $F_i^{2n} \subset D(x_i,R',2R)$ we have $d(c_v(0), (c_w(0)) \le 2 R^\prime$ and  $d(c_v(T_1), (c_w(T_2)) \le 4 R$ for all $v,w \in F_i^{2n} $ where
$2n = r(n) -2R - R' \le T_i \le r(n)  +2R +R' \le 2n + 4R +2R'$. Therefore,
\begin{align*}
d(c_v(2n), c_w(2n)) & \le d(c_v(2n), c_v(T_1) +d(c_v(T_1), c_w(T_2)) +d(c_w(T_2), c_w(2n))\\
& \le 12R +4R'
\end{align*}
Hence, Lemma  \ref{lem:endpts-suffice} implies $d(c_v(t), c_w(t)) \le 3 (12R +4R^\prime) + 4 \delta$ and therefore the 
sequence of coverings $A^n = \{\phi^n L_i^{2n}\}$ has the property stated in Lemma \ref{lem:sep-measures}  for
$b =  3 (12R +4R^\prime) + 4 \delta $.

\end{proof}
Furthermore, we will need the following standard lemma, whose proof we omit (see e.g. \cite{gK98}).

\begin{lemma}\label{4M}    
Let $a_1 \ldots a_k \ge 0$ and $\sum_{i=1}^{k} a_i \le 1$. Then
$$
- \sum_{i=1}^{k} a_i \log a_i \le \left(
\sum_{i=1}^{k} a_i\right) \log k + \frac{1}{e} \;.
$$
\end{lemma}
%%%%%%%%%%%%%
\subsection{Proof of Theorem \ref{thm:uniqueMME-inj}}
\begin{proof}
Assume that $\nu$ is a $\phi^t$-invariant measure with $h_{\nu} (\phi) = h(\phi) = h_{\mu} (\phi)$ but
 $\nu \not= \mu$. We want to show that this leads to a contradiction.
 Consider the affine decomposition
 $$
 \nu = \alpha \nu' + (1-\alpha) \mu'
$$ 
of  $\nu$ into invariant measures,
where $\nu'$ is singular and $\mu'$ is absolutely continuous with respect to $\mu$.
Since $\mu$ is ergodic we have $\mu = \mu'$. From 
$$
h = h_\nu(\phi) = \alpha h_{\nu'}(\phi) + (1-\alpha) h_\mu(\phi) =
\alpha h_{\nu'}(\phi) + (1-\alpha) h.
$$
we obtain $h_{\nu'} (\phi) = h$. 
Therefore, it is enough to show that there is no measure $\nu$ singular to $\mu$ such that $h_\nu(\phi) = h$.
Using the ergodic decomposition we an assume that $\nu$ is ergodic as well.
 
Since by assumption in Theorem  \ref{thm:uniqueMME-inj} we have that the injectivity radius of $M$ is larger than $16\delta$ 
 we can choose  $\rho >0$ such that $ r_0 = 4 \delta + 3\rho <\frac{\inj(M)}{4}$. 
 If the second condition holds we have that the geodesic flow is entropy expansive at a scale larger  than $ 16\delta$.
 In this case we additionally assume that $\rho >0$ is so small that the geodesic flow is entropy expansive at scale $ 4r_0$ as well.

Let $V = \{v_1 \ldots v_m\} \in SM$ be a 
maximal set $2 r_0$-separated with respect to $d_{2n}$.
%, where $r_0$ is as in  Lemma \ref{lem: r_0 separated}.
   Choose a partition 
$\BBB^n$ such that for each $B\in \BBB^n$ we have $B_{d_{2n}}(v_i,r_0) \subset B \subset 
B_{d_{2n}}(v_i,2r_0)$ for some $v_i \in V$. 
This yields $\diam _{d_{2n}}\BBB^n < 4r_0 < \inj(M)$. 
If the first condition holds, i.e. the entropy non-expansive measure is smaller then the topological entropy. then the ergodicity of $\nu$ yields $\nu(\EE) =1 $
Therefore Proposition \ref{prop:entexpansive}  implies: 
$$
h_{\nu} (\phi^{2n}) = h_{\nu} 
(\phi^{2n}, \BBB^n) 
$$
Theorem \ref{thm:Bowen} implies that this also holds under the second condition since by the above choice of $\rho$ the geodesic flow is entropy expansive at scale $4r_0$. 
Therefore, in both cases we obtain:
\[
2n h_{\nu} (\phi^1) = h_{\nu} (\phi^{2n}) = h_{\nu} 
(\phi^{2n}, \BBB^n) \le H_{\nu} (\BBB^n) = -\sum_{B\in \BBB^n} \nu(B) \log 
\nu(B).
\]
Since $\nu$ and $\mu$ are mutually singular we find 
a set $\Omega \subset SM$ such that $\mu(\Omega)=0$ and $\nu(SM \setminus \Omega)=0$. 
Using that  $\mu(SM \setminus \EE)=0$ we can assume that $SM \setminus \EE \subset \Omega$.
Consider the covering  $L^{2n} = \{L_1^{2n} \ldots L_k^{2n}\}$
 of $SM$ as introduced in Lemma \ref{lem:coveringL}.
and its push forward $\{\phi^n L_i^{2n}\} = A^n$.
By Lemmata \ref{lem:sep-measures} and \ref{4L} we find for a subsequence $n =n_j$ sets $C_{n}$ consisting of a union of 
elements of $A^{n}$ such that 
$(\mu+\nu)(C_n \Delta \Omega )\to 0$ as $n=n_j \to \infty$. In particular, 
$\mu(C_n) \to 0$ and $\nu(SM\setminus C_n) \to 0$ as $n =n_j\to\infty$. 
Using Lemma~\ref{4M} we obtain 
\begin{align*}
2n h_\nu(\phi) \le  & - 
\sum_{\{B\in \BBB^n \mid \phi^n B \cap C_n \not= \emptyset\}} 
\nu (B) \log \nu (B)
 - \sum_{\{B\in \BBB^n \mid \phi^nB \cap C_n = \emptyset\}} 
\nu(B) \log \nu (B)\\
\le &  \left(\sum_{\{B\in \BBB^n \mid \phi^nB \cap C_n\not=\emptyset\}} \nu(B) 
\right) \log 
\card\{B\in\BBB^n \mid\phi^n B  \cap C_n \not=\emptyset\}\\
& +\left(\sum_{\{B\in \BBB^n \mid \phi^nB \cap C_n = \emptyset\}} \nu(B)\right)
\log\card
\{B \in \BBB^n \mid \phi^nB \cap C_n = \emptyset\} + \frac{2}{e}\\
\le & b_n\log \card\{B\in\BBB^n \mid \phi^n B  \cap C_n \not=\emptyset\}
+ (1-b_n)\log\card \{\BBB^n\} +\frac{2}{e} \;,
\end{align*}
where 
$$
b_n = \sum\limits_{\{B\in \BBB^n \mid \phi^nB \cap C_n\not=\emptyset\}} \nu(B).
$$
Using Lemma \ref{lem: est-partition}  we estimate 
\begin{align*}
\card\{B\in \BBB^n \mid  &\phi^nB \cap C_n\not= \emptyset\} \\
&=  \sum\limits_{\{L_i^{2n}\in 
L^{2n} \mid \phi^n L_i^{2n} \subset C_n\}} 
\card\{B\in \BBB^n \mid \phi^n B \cap \phi^n L_i^{2n} \not= \emptyset\} \\
&\le a(\rho, R^\prime, 2R) \card\{L_i^{2n} \in L^{2n} \mid \phi^n L_i^{2n} \subset C_n\}.
\end{align*}
Furthermore, using Lemma \ref{lem:coveringL}    
$$
\card\{L_i^{2n} \mid \phi^n L_i^{2n} \subset C_n\} =
\card\{L_i^{2n} \mid  L_i^{2n} \subset \phi^{-n} C_n\} 
$$
$$
\le \alpha \frac{\mu(C_n)}{\min \mu (L_i^{2n})} \le \alpha^\prime \mu(C_n) 
 e^{2hn}
$$
for a constant $\alpha^\prime> 0$. Moreover, replacing $C_n$ by $SM$, the
same argument yields
$$
\card \{\BBB^n \} \le a(\rho,R^\prime, 2R) \alpha \frac{\mu(SM)}{\min \mu (\phi^nL_i^{2n})} 
\le  \tilde \alpha e^{2hn}
$$
for a suitable constant $\tilde \alpha > 0$.
Then $h_{\nu}(\phi) =h$ implies
$$
- \frac{2}{e} \le b_n 
\log \left(\mu(C_n)\ \alpha^\prime {a(\rho, ,R^\prime ,2R)}\right)
+ (1 - b_n) \log (\tilde \alpha).
$$
Since
$$
1 \ge b_n = \sum_{\{B \in \BBB^n \mid \phi^n B \cap C_n \not= \emptyset\}} 
\nu (\phi^nB) \ge \nu(C_n) \to 1
$$
and $\mu(C_n) \to 0$ the first term of the
right hand side tends to $-\infty$ as $n =n_j \to 
\infty$. Since the second term tends to 0 we obtain a contradiction.
\end{proof}
\subsection{Proof of uniqueness of the MME}\label{subsec: uniquenessMME}
Using Theorem  \ref{thm:uniqueMME-inj} we can prove Theorem \ref{thm:main} and Theorem \ref{thm:expansiveMME} stated in the introduction.
For that we have to guarantee Riemannian covers of $(M,g)$ with large injectivity radius. 
A sufficient for  that is that the fundamental group of $M$ is residually finite.
\begin{definition}\label{def:residuallyfinite}
A group $G$ is \emph{residually finite} if the intersection of its finite index subgroups is trivial.
\end{definition}

\begin{theorem}\label{prop:hyp-metric}
Every closed manifold of dimension less or equal to three  has residually finite fundamental group.
\end{theorem}
For surfaces this  was first proved by Baumslag \cite{gB62}. Later on, Hempel \cite{jH87} proved that fundamental groups of three manifolds are residually finite. It is an open problem whether every hyperbolic group is residually \cite{gA14}.

For our purposes we need the following implication of a manifold having residually finite fundamental group.
\begin{proposition}\label{prop:large-inj}
Let $M$ be a smooth Riemannian manifold and suppose that  $\pi_1(M)$ is residually finite.  Then for every $R>0$  there is a smooth Riemannian manifold $N$ and a locally isometric finite-to-1covering map $p\colon N\to M$ such that the injectivity radius of $N$ is at least $R$.
\end{proposition}
\begin{proof}
For a proof see \cite{CKW21}.
\end{proof}

Now we can prove Theorem  \ref{thm:main} and Theorem \ref{thm:expansiveMME} using Theorem \ref{thm:uniqueMME-inj}
as follows.
\begin{proof}
%{\bf Proof of Theorem \ref{thm:main} \ref{thm:expansiveMME}:}
Let $(M,g)$ be a closed manifold  Riemannian manifold without conjugate points and divergence property.
 Furthermore, assume that the fundamental group of $M$ is residually finite and Gromov hyperbolic fundamental group with Gromov hyperbolicity.
 By Proposition \ref{prop:large-inj} there is a smooth Riemannian manifold $N$ and is for some 
$k \in \N$ a locally isometric $k$ to $1$  covering map $p\colon N\to M$ such that the injectivity radius of $N$ is larger 
than $8 \delta$. In particular $N$ has no conjugate, the divergence property
and the  geodesic $\phi^t_{SM}: SM \to SM$ is $k$  to $1$ factor of the 
geodesic flow $\phi^t_{SN} : SN \to SN$ as well. If $\mu$ is a invariant Borel probability measure  on $SM$ 
we denote by $\tmu$ the canonical lift defined by
$$
\tmu(A) = \int_{SM} \frac{1} {k} \card \{dp^{-1} (v) \cap A \} d\mu(v),
$$
where $dp :SN \to SM$ is the differential of $p$.
Obviously is $\tmu$ is $\phi^t_{SN}$ invariant Borel probability measure with $dp_*\tmu =\mu$ and
$h_{\tmu}(\phi_{SN}) = h_{\mu}(\phi_{SM})$.
Since the expansive set of $SN$ is the lift of the expansive set in $SM$ 
and the expansivity constants of $\phi^t_{SM}: SM \to SM$ and  $\phi^t_{SN} : SN \to SN$ agree
the assumptions of Theorem \ref{thm:main} or Theorem \ref{thm:expansiveMME} imply that 
$\phi^t_{SN} : SN \to SN$ fullfils the assumptions
of Theorem \ref{thm:uniqueMME-inj}. Hence $\phi^t_{SN} : SN \to SN$ has a unique measure of maximal entropy.
 If $\mu_1$ and $\mu_2$ are measures of maximal entropy for  $\phi^t_{SM}$
then their lifts $\tmu_1$ and $\tmu_2$ are measures of maximal entropy for  $\phi^t_{SN}$ as well.
In particular $\tmu_1 = \tmu_2$  therefore $\mu_1 = \mu_2 = \mu$.

Since by Theorem \ref{thm:measure mu}
 $\tmu$  has full support, is mixing and $\tmu(\EE) =1$ this holds for the push-forward measure $\mu =dp_*\tmu$ as well.
  \end{proof}
  \bibliographystyle{plain}

\bibliography{geodflows}

\end{document}